\documentclass[12pt,reqno]{amsart}
\usepackage{amscd,amssymb,bbm,graphicx,times,url}

\usepackage{vmargin}
\setpapersize{USletter}
\setmargrb{1.5in}{.75in}{1.5in}{.75in}

\newtheorem{theorem}{Theorem}

\newtheorem{lemma}[theorem]{Lemma}

\newcommand{\N}{{\mathbb N}}

\newcommand{\R}{{\mathbb R}}

\newcommand{\supp}{\operatorname{supp}}

\newcommand{\Z}{{\mathbb Z}}

\begin{document}

\title[]{A note on the solution of the Mexican hat problem}

\author[]{H.-Q. Bui and R. S. Laugesen}
\address{Department of Mathematics, University of Canterbury,
  Christchurch 8020, New Zealand}
\email{Q.Bui\@@math.canterbury.ac.nz}
\address{Department of Mathematics, University of Illinois, Urbana,
IL 61801, U.S.A.} \email{Laugesen\@@illinois.edu}
\date{\today}



\begin{abstract}
We prove a technical estimate needed in our recent solution of the completeness question for the non-orthogonal Mexican hat wavelet system, in $L^p$ for $1<p<2$ and in the Hardy space $H^p$ for $2/3 < p \leq 1$.
\end{abstract}

\maketitle

\section{\bf Introduction}
\label{introduction}

Recently we solved the Mexican hat wavelet completeness problem \cite[\S8]{bl6}. Our proof relied on a certain technical estimate $\Delta_*(\Phi,\Psi) < 1$, which we prove in this note on the ArXiv.

We begin with some definitions. Let
\[
\Psi(\xi) = (2\pi \xi)^2 \exp(-2\pi^2 \xi^2) \qquad \text{and} \qquad  \Phi = \kappa/\Psi ,
\]
with $\kappa$ being the ``double bump'' function
\[
\kappa(\xi) =
\begin{cases}
0 , & \xi \in [0,1/12] , \\
\sin^2 \big( (12\xi - 1)\pi/2 \big) , & \xi \in [1/12,1/6] , \\
\cos^2 \big( (6\xi - 1)\pi/2 \big) , & \xi \in [1/6,1/3] , \\
0 , & \xi \in [1/3,\infty) , \\
\kappa(-\xi), & \xi \in (-\infty,0) .
\end{cases}
\]
(Note $\Psi$ is the Fourier transform of the Mexican hat function $\psi(x)=(1-x^2)e^{-x^2/2}$.) Put
\[
\Theta(\xi) = \xi \Phi^\prime(\xi) \qquad \text{and} \qquad \Gamma(\xi) = \xi \Phi(\xi) .
\]
Define
\begin{align*}
& \Delta(\Phi,\Psi) \\
& = \sum_{l \neq 0} \big\lVert \sum_{j \in \Z} |\Phi(\xi 2^{-j})
\Psi(\xi 2^{-j} - l)| \big\rVert_{L^\infty(\R)}^{1/2} \,
\big\lVert \sum_{j \in \Z} |\Phi(\xi 2^{-j} + l) \Psi(\xi 2^{-j})| \big\rVert_{L^\infty(\R)}^{1/2} ,
\end{align*}
and let
\[
\Delta_*(\Phi,\Psi) = \Delta(\Phi,\Psi) + 2\Delta(\Theta,\Psi) + 2\Delta(\Gamma,\Psi^\prime) .
\]

\section{\bf Proof that $\Delta_*(\Phi,\Psi) < 1$}
\label{examples}

We will prove $\Delta_*(\Phi,\Psi)<0.52$. If rigor is not required then the better numerical estimate $\Delta_*(\Phi,\Psi) < 0.03$ can be used. The purpose of this note is simply to demonstrate that a rigorous estimate can be obtained.

First we simplify the expression for $\Delta$.
\begin{lemma} \label{deltaexpression}
Assume $A$ and $B$ are measurable functions on $\R$. Suppose $A$ is supported in $[-1/3,-1/12] \cup [1/12,1/3]$, that $|A|$ and $|B|$ are even functions, and that $|B(\xi)|$ is decreasing for $\xi \geq 2/3$.
Then
\[
\Delta(A,B) \leq 2\sqrt{2} \lVert A(\xi)
B(1 - \xi) \rVert_{L^\infty[1/12,1/3]} + 2\sqrt{2} \lVert A \rVert_{L^\infty[1/12,1/3]} \sum_{l=2}^\infty
|B(l-1/3)| .
\]
\end{lemma}
\begin{proof}
We start by noting
\begin{equation} \label{Best}
|B(l + \xi)| \leq |B(l - \xi)| \qquad \text{whenever $\xi \in [1/12,1/3], \quad l \in \N$,}
\end{equation}
because $l+\xi > l-\xi \geq 1-1/3=2/3$ and $|B|$ is decreasing on $[2/3,\infty)$.

Now consider $l \neq 0$. The support hypothesis on $A$ implies that
\begin{align}
& \big\lVert \sum_{j \in \Z} |A(\xi 2^{-j}) B(\xi 2^{-j} - l)| \big\rVert_{L^\infty(\R)} \notag \\
& = \lVert |A(\xi)
B(\xi-l)| + |A(\xi/2) B(\xi/2 - l)| \rVert_{L^\infty([-1/3,-1/6] \cup [1/6,1/3])} \notag \\
& \leq 2 \lVert A(\xi) B(\xi-l) \rVert_{L^\infty([-1/3,-1/12] \cup [1/12,1/3])} \notag \\
& \leq 2 \max_\pm \lVert A(\xi) B(|l| \pm \xi) \rVert_{L^\infty[1/12,1/3]} \qquad \text{by evenness of $|A|$ and $|B|$} \notag \\
& = 2 \lVert A(\xi) B(|l| - \xi) \rVert_{L^\infty[1/12,1/3]} \label{simple1}
\end{align}
by \eqref{Best}.

Next we claim the sets $\{ (\supp(A) - l)2^j \}_{j \in \Z}$ are disjoint. When $l<0$,
\[
\supp(A)-l \subset \big[ |l|- \frac{1}{3}, |l|+
\frac{1}{3} \big] ,
\]
and the left endpoint of this last interval dilates under multiplication by $2$ to the right of the right endpoint,
because $2(|l|-1/3) \geq |l|+1/3$; argue similarly for disjointness when $l>0$.

The disjointness ensures that
\begin{align}
\big\lVert \sum_{j \in \Z} |A(\xi 2^{-j} + l) B(\xi 2^{-j})| \big\rVert_{L^\infty(\R)}
& = \lVert A(\xi + l) B(\xi) \rVert_{L^\infty(\supp(A)-l)} \notag \\
& = \lVert A(\xi) B(\xi-l) \rVert_{L^\infty(\supp(A))} \notag \\
& = \lVert A(\xi) B(|l|-\xi) \rVert_{L^\infty[1/12,1/3]} \label{simple2}
\end{align}
by evenness of $|A|$ and $|B|$ and estimate \eqref{Best}.

By putting the estimates \eqref{simple1} and \eqref{simple2} into the definition of $\Delta(A,B)$, we conclude that
\begin{align*}
\Delta(A,B) & \leq 2\sqrt{2} \sum_{l = 1}^\infty \lVert A(\xi) B(l-\xi) \rVert_{L^\infty[1/12,1/3]} .
\end{align*}
The lemma now follows by splitting off the term with $l=1$ and using that $|B|$ is decreasing on $[2/3,\infty)$.
\end{proof}

Next we state some calculus facts about the function $\Psi(\xi) = (2\pi \xi)^2 \exp(-2\pi^2 \xi^2)$.
\begin{lemma} \label{fact1}
$|\Psi|$ and $|\Psi^\prime|$ are decreasing for $\xi \in [2/3,\infty)$. (Hence $\Psi$ and $\Psi^\prime$ satisfy the hypotheses on ``$B$'' in Lemma~\ref{deltaexpression}.)
\end{lemma}
\begin{lemma} \label{fact2}
Let $m,n \in \{ 0,1,2,3 \}$. Then $\xi^{-m}(1-\xi)^n e^{4\pi^2 \xi}$ is increasing for $\xi \in [1/12,1/3]$.
\end{lemma}

Now we estimate the three terms in $\Delta_*(\Phi,\Psi)$.

\subsubsection*{Estimation of $\Delta(\Phi,\Psi)$.} We have $|\kappa| \leq 1$ and
\begin{align}
\Phi(\xi) & = \frac{\kappa(\xi)}{\Psi(\xi)}=\kappa(\xi) (2\pi \xi)^{-2} e^{2\pi^2 \xi^2} , \notag \\
\Psi(1-\xi) & = (2\pi)^2 e^{-2\pi^2} (1-\xi)^2 e^{4\pi^2 \xi} e^{-2\pi^2 \xi^2} , \label{psieq}
\end{align}
so that (by using Lemma~\ref{fact2} and evaluating at $\xi=1/3$)
\begin{equation}
|\Phi(\xi) \Psi(1-\xi)| < 0.006 , \qquad \xi \in [1/12,1/3] . \label{est2}
\end{equation}

Further, for $l \geq 2$ we have
\[
|\Psi(l-1/3)| < (2\pi)^2 l^2 e^{-2\pi^2 (l/2)^2} \leq (2\pi)^2 2^2 e^{l-2} e^{-\pi^2 l} ,
\]
so that by a geometric series,
\begin{equation} \label{psiltwo}
\sum_{l=2}^\infty |\Psi(l-1/3)| < (2\pi)^2 4 e^{-2\pi^2} / (1 - e^{1-\pi^2}) .
\end{equation}
Combining \eqref{psiltwo} with the fact that
\[
|\Phi(\xi)| < 200(2\pi)^{-2} , \qquad \xi \in [1/12,1/3] ,
\]
gives that
\[
\lVert \Phi \rVert_{L^\infty[1/12,1/3]} \sum_{l=2}^\infty
|\Psi(l-1/3)| < 0.000003 .
\]

Substituting this last estimate and \eqref{est2} into Lemma~\ref{deltaexpression} shows that
\begin{equation} \label{est3}
\Delta(\Phi,\Psi) < 0.02 .
\end{equation}

\subsubsection*{Estimation of $\Delta(\Theta,\Psi)$.} By definition of $\Phi=\kappa/\Psi$, we have
\begin{align}
|\Theta(\xi)| & = |\xi \Phi^\prime(\xi)| \notag \\
& \leq (2\pi)^{-2} e^{2\pi^2 \xi^2}
\begin{cases}
6\pi \xi^{-1} + 2 \xi^{-2} & \text{when $\xi \in [1/12,1/6]$} \\
3\pi \xi^{-1} + \big( 4\pi^2(1/3)^2 - 2 \big) \xi^{-2} & \text{when $\xi \in [1/6,1/3]$}
\end{cases} \label{thetaeq} \\
& < (2\pi)^{-2} \cdot 600 . \notag
\end{align}
Multiplying this last estimate by \eqref{psiltwo} shows
\begin{equation} \label{est4}
\lVert \Theta \rVert_{L^\infty[1/12,1/3]} \sum_{l=2}^\infty
|\Psi(l-1/3)| < 0.000007 .
\end{equation}

Using \eqref{psieq}, \eqref{thetaeq} and Lemma~\ref{fact2} gives that
\begin{equation}
|\Theta(\xi) \Psi(1-\xi)| < 0.031 , \qquad \xi \in [1/12,1/3]. \label{est5}
\end{equation}

Substituting \eqref{est4} and \eqref{est5} into Lemma~\ref{deltaexpression} shows that
\begin{equation} \label{est6}
\Delta(\Theta,\Psi) < 0.09 .
\end{equation}

\subsubsection*{Estimation of $\Delta(\Gamma,\Psi^\prime)$.}
Recall the definition
\[
\Gamma(\xi) = \xi \Phi(\xi)=\kappa(\xi) (2\pi)^{-2} \xi^{-1} e^{2\pi^2 \xi^2} .
\]
From
\[
\Psi^\prime(\xi) = 2(2\pi)^2 (\xi-2\pi^2 \xi^3) e^{-2\pi^2 \xi^2}
\]
we find for $\xi < 1$ that
\[
|\Psi^\prime(1-\xi)| \leq 2(2\pi)^2 e^{-2\pi^2} \big( (1-\xi)+2\pi^2 (1-\xi)^3 \big) e^{4\pi^2 \xi} e^{-2\pi^2 \xi^2} .
\]
Hence (by Lemma~\ref{fact2} and evaluating at $\xi=1/3$)
\begin{equation}
|\Gamma(\xi) \Psi^\prime(1-\xi)| < 0.055 , \qquad \xi \in [1/12,1/3] . \label{est8}
\end{equation}

Next,
\[
|\Psi^\prime(\xi)| \leq (2\pi)^4 \xi^3 e^{-2\pi^2 \xi^2} , \qquad \xi \geq 1 .
\]
Hence for $l \geq 2$,
\[
|\Psi^\prime(l-1/3)| \leq (2\pi)^4 l^3 e^{-2\pi^2 (l/2)^2} \leq (2\pi)^4 3^3 e^{l-3} e^{-\pi^2 l} ,
\]
so that by a geometric series,
\[
\sum_{l=2}^\infty |\Psi^\prime(l-1/3)| \leq 27 (2\pi)^4 e^{-1-2\pi^2} / (1 - e^{1-\pi^2}) .
\]
Combining this last estimate with the fact that
\[
|\Gamma(\xi)| < 30(2\pi)^{-2} , \qquad \xi \in [1/12,1/3] ,
\]
gives that
\begin{equation} \label{est9}
\lVert \Gamma \rVert_{L^\infty[1/12,1/3]} \sum_{l=2}^\infty
|\Psi^\prime(l-1/3)| < 0.00004 .
\end{equation}

Substituting \eqref{est8} and \eqref{est9} into Lemma~\ref{deltaexpression} shows that
\begin{equation} \label{est10}
\Delta(\Gamma,\Psi^\prime) < 0.16 .
\end{equation}

\subsubsection*{Estimation of $\Delta_*(\Phi,\Psi)(\Phi,\Psi)$.}
We obtain that
\[
\Delta_*(\Phi,\Psi) = \Delta(\Phi,\Psi) + 2\Delta(\Theta,\Psi) + 2\Delta(\Gamma,\Psi^\prime) < 0.52 ,
\]
by summing estimates \eqref{est3}, \eqref{est6} and \eqref{est10}. The proof is complete.

\end{document}